\documentclass[12pt,twoside]{article}
\usepackage{inputenc}
\usepackage[english]{babel} 
\usepackage{amsfonts}
\usepackage{amsmath}
\usepackage{amssymb}
\usepackage{amsthm}
\usepackage{graphicx}          
\usepackage{subfig}
\RequirePackage{amsthm}
\theoremstyle{plain} 
\newtheorem{theorem}{Theorem}
\newtheorem{corollary}{Corollary}
\newtheorem{lemma}{Lemma}
\newtheorem{remark}{Remark}
\newtheorem{example}{\bf Example}[section]

\theoremstyle{definition}
\newtheorem{definition}{Defintion}

\theoremstyle{remark}

\begin{document}
\def\Re{\mathop{\rm Re}\,}
\def\Im{\mathop{\rm Im}\,}
\def\dom{\mathop{\rm dom}\,}
\def\dist{\mathop{\rm dist}}
\def\grad{\mathop{\rm grad}}
\renewcommand{\proof}{\vspace{2mm}\hspace{-7mm}\textit{Proof.}}
\renewcommand{\endproof}{\begin{flushright} \vspace{-2mm}$\Box$\vspace{-4mm}
\end{flushright}}
\newcommand{\phan}{\hspace*{0cm}}
\newcommand{\comment}{}

\author{I. R. Muftahov, D. N. Sidorov and N. A. Sidorov}
\title{On Perturbation Method for the First Kind Equations:\\
  Regularization and Application }
\maketitle{}
\begin{abstract}
One of the most common problems of scientific applications is computation of the derivative of a function specified by possibly noisy or imprecise experimental data. 
Application of  conventional techniques for numerically calculating derivatives will amplify the noise making the result useless. We address this typical ill-posed problem by application of  perturbation method to linear first kind equations
$Ax=f$ with bounded operator $A.$ We assume that we know  the operator $\tilde{A}$ and source function $\tilde{f}$ only such as  $||\tilde{A} - A||\leq \delta,$ $||\tilde{f}-f||< \delta.$ 
The regularizing equation
$\tilde{A}x + B(\alpha)x = \tilde{f}$  possesses the unique solution. Here $\alpha \in S,$
 $S$ is assumed to be an open space in   $\mathbb{R}^n,$ $0 \in \overline{S},$
$\alpha= \alpha(\delta).$
 As result of proposed theory, we suggest a novel algorithm providing accurate results even in the presence of a large amount of noise. \\

\end{abstract}
\markboth{I. R. Muftahov, D.N. Sidorov and N.A. Sidorov}{On Perturbation Method for the First Kind Equations:
  Regularization and Application}


\section*{Introduction}
\hspace{0.7 cm}

Let $A$ be bounded operator in banach space
$X$  with range  ${{R}}(A)$
in a banach space $Y.$ Let us consider the following linear operator equation
\begin{equation}
Ax = f,\,\,\, f \in {R}(A).  \label{eq1}
\end{equation}
We assume  domain  ${R}(A)$ can be non closed and $\mathrm{Ker} \,A \neq \{0\}$. 
In many practical problems one needs to solve an approximate equation
\begin{equation}
\tilde{A}x = \tilde{f},  \label{eq2}
\end{equation}
instead of exact equation. Here
 $\tilde{A}$ and $\tilde{f}$ are approximations of exact operator $A$ and right hand side function $f$ correspondingly such as 
\begin{equation}
||\tilde{A}-A|| \leq  \delta_1, \,\,\, ||\tilde{f}-f|| \leq \delta_2, \,\,\, \delta=\max\{\delta_1,\delta_2\}.  \label{eq3}
\end{equation}
The problem of solution of the equation (\ref{eq2}) is ill-posed and therefore unstable even with respect to small errors and it needs regularization in real world numerous applications. 
The basic results in regularization theory and methods for solution of the inverse
problems have been gained in scientific schools of A.~N.~Tikhonov, V.~I.~Ivanov and 
M.~M.~Laverentiev. Nowadays, this field of contemporary mathematics 
promotes the developments of many interdisciplinary fields in science and technologies
 \cite{hao2012, iv1978,  lat1970, lav1962,  bak1968, tih1974, yag2014}.
  There were many efficient regularization methods have been proposed for operator equation
(\ref{eq1}). The most efficient regularization methods are Tikhonov's method  of stabilizer functional, the quasi-solution method suggested by V.~K.~Ivanov,  M.~M.~Laverentiev perturbation method,  V.~A.~Morozov's discrepancy principle
and other methods. The principal role in the theory plays variational approaches, spectral theory, perturbation theory and functional analysis methods. 
V.P.~Maslov (here readers may refer to \cite{mas1968}) has 
established the equivalence of ill-posed problem solution existence and convergence of the  regularization process. 
There is the constant interest of   regularization methods to be applied  in interdisciplinary
research and applications related to signal and image processing,
numerical differentiation and inverse problems.

In this article we will address the  regularized processes construction by 
 introduction of the following perturbed equation
\begin{equation}
{A}x_{\alpha} +B(\alpha)x_{\alpha} = {f}. \label{eq4}
\end{equation}
In this paper we continue and upgrade the results  \cite{iv1978}, \cite{sid1976}, \cite{sid1982}.
It is to be noted that regularization method based on perturbed equation was first
proposed by M.~M.~Laverentiev \cite{lav1962} in case of completely continuous 
self-adjoint and positive operator $A$ and $B(\alpha) \equiv \alpha$.

Following \cite{sid2002} we select the stabilizing operator (SO) $B(\alpha)$
to make solution
$x_{\alpha}$ unique and provide computations stability.
Let us call   $\alpha  \in  S \subset \mathbb{R}^n$ as vector parameter of regularization.
Here $S$ is an open set,  zero belongs to the boundary of this set
(briefly, S-sectoral neighborhood of zero in  $\mathbb{R}^n$),
$\lim\limits_{S\ni \alpha \rightarrow 0} B(\alpha) = 0.$
Parameter $\alpha$
we  adjust to the data error level $\delta$.
Similar approach was suggested in the monographs \cite{sid1982}, \cite{sid2002},
but in this article the regularization parameter $\alpha$
can be vector. 
Previously only the simple case has been addressed with
$B(\alpha) = B_0 +\alpha B_1, \, \alpha \in \mathbb{R}^+$. 
Such SO has been employed in the development and justification of iterative
 methods of Fredholm points $ \lambda_0$ calculation,
zeros and the elements of the generalized Jordan sets of operator functions \cite{log1976, sid1978},
for the construction of approximate methods in the theory of branching of solutions of nonlinear
operator equations with parameters \cite{sid2002, sid1995, sid2010, sid2010a},
of  construction of solutions of differential-operator
equations with irreversible operator coefficient in the main part of  \cite{sid2002}.
In present article we propose  the novel  theory for
 operator systems regularization. 


The paper is organized as follows.
In Sec. 1 we obtained the sufficient conditions when
perturbed equation (\ref{eq4}) enables a regularization process.
In Sec. 2 we suggested the choice of SO $ B(\alpha). $
An important role is played by a classic theorem of Banach -- Steinhaus.
In Sec. 3 we consider the application of regularizing equation of the form (\ref{eq4})
  in the problem of stable differentiation.

\section{The fundamental theorem of regularization by the perturbation method}

Apart from equations (\ref{eq1}), (\ref{eq2}), (\ref{eq4}) let us introduce the equations
\begin{equation}
({A}x +B(\alpha))x = \tilde{f}, \label{eq11}
\end{equation}
\begin{equation}
(\tilde{A}x +B(\alpha))x = \tilde{f}. \label{eq12}
\end{equation}
Operator $B(\alpha)$  errors can be always included into the operator $\tilde{A}.$ 
 Equation (\ref{eq12}) we call regularized equation (RE) 
for the problem (\ref{eq2}). The following estimates are assumed to be fulfilled below
\begin{equation}
||({A} +B(\alpha))^{-1}|| \leq c(|\alpha|), \label{eq13}
\end{equation}
\begin{equation}
||B(\alpha)|| \leq d(|\alpha|), \label{eq14}
\end{equation}
where $c(|\alpha|)$ is continuous function, $\alpha \in S \subset \mathbb{R}^n,$
$0\in \overline{S},$
$ \lim\limits_{|\alpha| \rightarrow 0} c(|\alpha|) = \infty, \,\,   \lim\limits_{|\alpha| \rightarrow 0} d(|\alpha|) =0.$
If $x^*$ is the solution to  equation (\ref{eq1}), then  
$(A+B(\alpha))^{-1} f - x^* =  (A+B(\alpha))^{-1} B(\alpha) x^*.$
Therefore, we have
\begin{lemma} \label{lem1}
Let $x^*$ be some solution to the equation (\ref{eq1}), $x(\alpha)$ satisfy the equation (\ref{eq4}).  Then, in order to  $x_{\alpha} \rightarrow x^*$ for
$S \ni \alpha  \rightarrow 0$ it is necessary and sufficient to have the following 
equality fulfilled
\begin{equation}
S(\alpha,x^*) = ||(A + B(\alpha))^{-1} B(\alpha) x^*|| \rightarrow 0 \,\,\text{for} \,\,\, S \ni \alpha \rightarrow 0. \label{eq15}
\end{equation}
  \end{lemma} 
In \cite{leon13} there are sufficient conditions for ensuring  
estimates (\ref{eq13}) -- (\ref{eq14}), and examples addressing case of vector parameter. 
Application of such estimates for solving nonlinear equations are also considered.
Let us follow \cite{sid1982} and introduce the following definition.
\begin{definition} \label{def1}
We call the condition (\ref{eq15}) as {\it stabilization condition}, operator $B(\alpha),$
we call  {\it stabilization operator} if it satisfy the condition  (\ref{eq15}), and
solution $x^*$ we call {\it $B$-normal solution} of equation (\ref{eq1}).
\end{definition}
\begin{remark} 
Obviously  the  limit of the sequence $ \{x_{\alpha} \} $ is unique one in normed space  and therefore the equation (\ref{eq1}) can have only one  $B$-normal solution.
\label{vipre1}
\end{remark}
 From estimates (\ref{eq13}) -- (\ref{eq14}) it follows 
\begin{lemma}
Let  $x_{\alpha}$ and $\hat{x}_{\alpha}$ be solutions of the
equations  (\ref{eq4})   and
(\ref{eq11}) correspondingly.
 If parameter $\alpha = \alpha(\delta) \in S$ is selected such as $\delta \rightarrow 0$
\begin{equation}
|\alpha(\delta)| \rightarrow 0\,\, \text{and} \,\, \delta c(|\alpha(\delta)|) \rightarrow 0,
\label{eq16}
\end{equation}
 then $\lim\limits_{\delta \rightarrow 0} ||x_{\alpha} - \hat{x}_{\alpha} || = 0.$
\label{lem2}
\end{lemma}
\begin{definition}
Condition (\ref{eq16}) we  call the coordination condition of vector
parameter $\alpha$  with error level $\delta.$
\label{def12}
\end{definition}
The coordination conditions play principal role in all regularization methods
for  ill-posed problems (here readers may refer e.g. to \cite{iv1978, bak1968, lav1962,sid1980, sid1982,  mar1973, sid2002,  tren1977}).
The coordination condition is assumed to be fulfilled. Below we also assume 
  $ \alpha $ depends on $ \delta $ but for the sake of brevity we omit this fact.
\begin{lemma}
Let the estimates (\ref{eq13}) -- (\ref{eq14}) be satisfied as well as coordination condition 
for the regularization parameter
(\ref{eq16}). Next we select $q \in (0,1)$
 and find $\delta >0$ such as for $\delta \leq \delta_0$  the following inequality
\begin{equation}
\delta  c(|\alpha|)) \leq q
\label{eq17}
\end{equation}
will be fulfilled.
Then   $\tilde{A} +B(\alpha)$   is continuously invertible operator and the following estimates
are fulfilled
\begin{equation}
||(\tilde{A}+B(\alpha))^{-1} || \leq \frac{||(A+B(\alpha))^{-1}||}{1-q},
\label{eq18}
\end{equation}
\begin{equation}
||(\tilde{A}+B(\alpha))^{-1} f || \leq ||(A+B(\alpha))^{-1} f|| + \delta \frac{c(|\alpha|)}{1-q}
||(A+B(\alpha))^{-1} f||.
\label{eq19}
\end{equation}
\label{lem3}
\end{lemma} 
\begin{proof}
Based on  estimate (\ref{eq3})  for $\forall f$
 we have
\begin{equation}
||(\tilde{A} - A)(A+B(\alpha))^{-1} f||  
\leq \delta  ||(A+B(\alpha))^{-1} f||.
\label{eq110}
\end{equation}
Hence taking into account the estimates (\ref{eq13}), (\ref{eq14}), (\ref{eq17}), 
we have the following inequality
\begin{equation}
||(\tilde{A}-A) (A+B(\alpha))^{-1} f|| \leq \delta  c(|\alpha|)  \leq q ||f ||.
\label{eq111}
\end{equation}
Now since $q\leq 1$ we have
$\tilde{A} +B(\alpha) = (I +(\tilde{A} - A) (A+B(\alpha))^{-1}) (A+B(\alpha)),$
then existence of inverse operator $(\tilde{A} + B(\alpha))^{-1},$
as well as estimate (\ref{eq18}) follows from known inverse operator  Th.. 
Next we employ the following operator
identity $C^{-1} = D^{-1}-D^{-1}(I+(C-D)D^{-1})^{-1}(C-D) D^{-1}$
where $C=(\tilde{A}+B(\alpha)),\, D=A+B(\alpha)$
and, based on inequalities (\ref{eq110}), (\ref{eq111}) we get estimate (\ref{eq19}).
\end{proof}
\begin{theorem}[Main  Theorem]  
 Let conditions of Lemma  \ref{lem3} be fulfilled, i.e
 parameter $\alpha$ is coordinated  with noise level $\delta$. 
Then RE
(\ref{eq12}) has a unique solution $\tilde{x}_{\alpha}.$
Moreover if in addition  $x^*$ is solution of the exact  equation (\ref{eq1}),
then the following estimate is fulfilled
\begin{equation}
||\tilde{x}_{\alpha}-x^*|| \leq S(\alpha,x^*) +\frac{\delta c(|\alpha|)}{1-q}
\biggl(1+ ||x^*||  +
  S (\alpha, x^*) \biggr).
\label{eq113}
\end{equation} 
If also  $x^*$ is  $B$-normal solution of the equation 
(\ref{eq1}) then  $\{\tilde{x}_{\alpha}\}$  converges to  $x^*$  at a rate determined by bound (\ref{eq113}) as $\delta \rightarrow 0$.

   \label{th1}
\end{theorem}

\begin{proof}
Existence and uniqueness of the sequence $\{ \tilde{x}_{\alpha}\}$
as solution of RE (\ref{eq12}) for $\alpha \in S$
proved in  Lemma \ref{lem3}. Since
$(\tilde{A} +B(\alpha))(\tilde{x}_{\alpha} - x^*) = \tilde{f} - f- (\tilde{A}-A) x^* -B(\alpha)x^*, $
then we get the desired bound  (\ref{eq113})
$||\tilde{x}_{\alpha} - x^*|| \leq ||(\tilde{A} +B(\alpha))^{-1}||  ( ||\tilde{f}-f||+
||(\tilde{A}-A)x^*|| + ||(\tilde{A} + B(\alpha))^{-1} B(\alpha) x^*|| )
\leq S(\alpha, x^*) + \frac{\delta c(|\alpha|)}{1-q}  \biggl( 1  + ||x^*||  + S(\alpha, x^*) \biggr)$
based on the proved estimates (\ref{eq18}), (\ref{eq19})
and  (\ref{eq14}).
Since $x^*$ is $B$-normal solution then  $\lim\limits_{\alpha \rightarrow 0}
S(\alpha,x^*) = 0.$  And thanks to parameter $\alpha$ coordinated with noise level 
 $\delta$ we have  $\lim\limits_{\delta \rightarrow 0} \delta c (|\alpha|) = 0.$ 
Hence, due to bound  (\ref{eq113})
$\lim\limits_{\delta \rightarrow 0} ||\tilde{x}_{\alpha}-x^*||=0$
which completes the proof.
\end{proof}
As footnote of the section it's to be mentioned that for practical applications of this  theorem one needs
recommendations on the choice of SO $ B (\alpha) $ and 
$B$-normal solution existence conditions. It's also useful to know the necessary and sufficient
conditions of the existence of  $B$-normal solutions $ x^* $ to the exact equation (\ref{eq1}).~These~issues~we~discuss~below.
\section{Stabilizing  operator $B(\alpha)$ selection,  
   $B$-normal solutions existence and correctness class of  problem
 (\ref{eq1})}
If $A$ is Fredholm operator, $\{\phi_i\}_1^n$ is basis in  $\mathcal{N}(A),$  $\{\psi_i \}_1^n$ is basis in $\mathcal{N}^*(A),$
then (here readers may refer to  Sec. 22 in  textbook  \cite{trenogin}),
one may assume $B(\alpha) \equiv \sum\limits_{i=1}^n \langle \cdot , \gamma_i \rangle z_i,$
where $\{\gamma_i \}, \, \{z_i \}$ are selected such as 
$\det [ \langle \phi_i, \gamma_k \rangle ]_{i,k=1}^n  \neq 0,$
$\langle z_i, \psi_i \rangle = \begin{cases} 1 &\mbox{if } i = k \\ 
0 & \mbox{if } i \neq k \end{cases} $
 herewith
the equation
\begin{equation}
Ax = f - \sum\limits_{i=1}^n \langle f, \psi_i \rangle z_i
\label{eqq21}
\end{equation}
is resolvable for arbitrary source function $f.$
Let us now recall $\tilde{f},$ which is $\delta$-approximation of  $f.$
     Then perturbed equation 
${A}x + \sum\limits_{i=1}^n \langle x, \gamma_i \rangle z_i = \tilde{f} - \sum\limits_{i=1}^n \langle \tilde{f}, \psi_i \rangle z_i$
has unique solution $\tilde{x}$ such as $||\tilde{x} - x^*|| \rightarrow 0$
 for $\delta \rightarrow 0,$ where $x^*$ is unique solution of exact solution (\ref{eqq21})
for which 
$\langle x^*, \gamma_i \rangle = 0, \, \, i=\overline{1,n}. $
Thus, in the case of a Fredholm operator $ A $ as a stabilizing
 operator one can take finite-dimensional operator
$B=\sum\limits_1^n\langle \cdot, \gamma_i \rangle z_i$
 which does not depend on the parameter $ \alpha.$
That is the regularization of iterative methods we employed in our papers
\cite{sid2010, sid2010a}, \cite{sid2002,sid2014} for the second order
nonlinear equations studies with parameters.
Of course, with this choice of SO $B$ it is required to have information about the kernel of the operator $ A $
and its defect subspace. Therefore, it is of interest to give recommendations on the choice  of SO $ B (\alpha) $ without the use of such information.
It is important to consider in a more complex problem solving the first kind equations,
when the range of the operator $ A $ is not closed.
It is to be noted that in papers
\cite{sid1976, sid1980} and in the monograph \cite{sid1982, sid2002}  we 
constructed the SO for the first kind equations as $B_0 + \alpha B_1$
where $\alpha \in \mathbb{R}^+.$
Below we consider the generalization of such results when  $B=B(\alpha),$
 $\alpha \in S \subset \mathbb{R}^n.$
Our previous results presented in papers
\cite{sid1980, sid1976, sid1982} follows from the proved  Th. 2 and 3 as
special cases.
 
\begin{theorem}
Let
$||(A+B(\alpha))^{-1}|| \leq c(|\alpha|), \, ||B(\alpha)||\leq d(|\alpha|)$
for
$\alpha \in S \subset \mathbb{R}^n,$
$c(|\alpha|), \, d(|\alpha|)$ are continuous functions, $\lim\limits_{|\alpha| \rightarrow 0}
c(|\alpha|) = \infty,$ $\lim\limits_{|\alpha| \rightarrow 0} d(|\alpha|) = 0.$
Let $\lim\limits_{|\alpha|\rightarrow 0} c(|\alpha|) d (|\alpha|) < \infty,$
$\mathcal{N}(A) = {0},$  $\overline{{R(A)}}=Y.$
Then unique solution $x^*$ of equation (\ref{eq1})    
is  $B$-normal solution and operator $B(\alpha)$ is its SO.
\label{th2}
\end{theorem}

\begin{proof}
First, lets  $B(\alpha) x^* \in \mathbb{R}(A)$ for $\alpha \in S.$
Then exists element
 $x_1(\alpha)$  such as $A x_1(\alpha) = B(\alpha) x^*.$
Then 
$(A+B(\alpha))^{-1} B(\alpha) x^* = (A+B(\alpha))^{-1} (A x_1(\alpha) + B(\alpha) x_1(\alpha) - B(\alpha)x_1(\alpha) ) 
= x_1(\alpha) - (A+B(\alpha))^{-1} B(\alpha) x_1(\alpha).$
Since $B(0)=0,$ $\mathcal{N}(A)=\{0\}$ then $\lim\limits_{S\ni \alpha \rightarrow 0} x_1(\alpha) = 0.$
It is to be noted that by condition $||(A+B(\alpha))^{-1} B(\alpha)|| \leq c(|\alpha|) d (|\alpha|),$
where  $c(|\alpha|) d (|\alpha|)$ is continuous function such as the limit 
$\lim\limits_{|\alpha| \rightarrow 0} c(|\alpha|) d (|\alpha|)$ is finite.
Then $\alpha$-sequence $\{ ||(A+B(\alpha))^{-1}B(\alpha) x^* || \}$
infinitesimal when $S \ni \alpha \rightarrow 0.$
The sequence of operators  $\{ (A+B(\alpha))^{-1} B(\alpha) \}$
converges pointwise to the zero operator on the linear manifold
 $L_0 = \{ x \,|\, B(\alpha) x \in {R}(A) \}.$
Thus, we have proved that the  Th. is true when
$B(\alpha) x^* \in {R}(A).$
Because by condition $\sup\limits_{\alpha \in S} c(|\alpha|) d(|\alpha|) < \infty$
then $\alpha$-sequence $\left\{ ||(A+B(\alpha))^{-1} B(\alpha)|| \right\}$
 is bounded. Therefore, the
sequence of linear operators
$\{ (A+B(\alpha))^{-1} B(\alpha) \}$  in space  $X$
converges pointwise to the zero operator on the linear manifold
$L_0 = \{ x \bigl | B(\alpha)  x \in {{R}(A)} \}.$
But then, on the basis of the Banach -- Steinhaus  Theorem we
have pointwise convergence of this operators sequence   to the zero operator
and on the closure
$\overline{L_0},$
i.e. when $B(\alpha) x^* \in \overline{{R}(A)}.$
 Since $\overline{{R}(A)}=Y$ and $B(\alpha) \in \mathcal{L}(X \rightarrow Y),$
then  $B(\alpha)x^* \in Y$ and Th. 2 is proved.

\end{proof}
The conditions of  Th. 2 can be relaxed
(here readers may refer to Cor. 1 and  Th. 3).
\begin{corollary}
If $\overline{{R}(A)} \subset Y,$ $\lim\limits_{S\ni \alpha \rightarrow 0} x_1(\alpha) =0$  then solution $x^*$ of exact equation (\ref{eq1})
is $B$-normal iff $B(\alpha)x^* \in \overline{{R}(A)}.$
\end{corollary}
\begin{remark}
In Cor. 1  condition  $\mathcal{N}(A) = \{ 0\}$  is not used.
The set $L = \{ x | B(\alpha) x \in \overline{{R}(A)} \}$
in conditions of Cor.1 defines maximum correctness class.
\end{remark}
It is to be noted that in  Th. 2 we used the assumption on
the finite limit
$\lim\limits_{S\ni \alpha \rightarrow 0} ||(A+B(\alpha))^{-1}||  ||B(\alpha)||.$ 
We can also relax this limitation.

\begin{theorem}
Let $||(A+\alpha B)^{-1}|| \leq c(\alpha),$
where $\alpha \in {\mathbb R}^1,$ $c(\alpha): (0, \alpha_0] \rightarrow \mathbb{R}^+$ is continuous function.
Suppose that there is a positive integer $n \geq 1$ such as  $\lim\limits_{\alpha \rightarrow 0} c(\alpha) \alpha^i = \infty, \, i=\overline{0,n-1},$
$\lim\limits_{\alpha \rightarrow 0} c(\alpha) \alpha^n < \infty$. 
Let $x_0$  satisfies the equation (\ref{eq1}) and in case of $n\geq 2$
there exist $x_1, \cdots , x_{n-1}$  which satisfy the sequence of equations
$
A x_i = B x_{i-1},\,\, i=1,\cdots , n-1.
$
Then $x_0$  is $B$-normal solution to equation (\ref{eq1})
iff $B x_{n-1} \in \overline{{R}(A)}.$
\label{th2}
\end{theorem}
\begin{proof}
Since  $A x_i = B x_{i-1}$ we have an equality
$(A+\alpha B)^{-1} \alpha B x_0  = \alpha (A+\alpha B)^{-1} (Ax_1 +\alpha B x_1 - \alpha B x_1) 
=\alpha x_1 -\alpha^2 (A+\alpha B)^{-1} B x_1 = \cdots = \alpha x_1 -\alpha^2 x_2 + \cdots -  (-1)^n \alpha^n (A+\alpha B)^{-1} B x_{n-1}, $
where if $\alpha \rightarrow 0$  then first $n-2$ terms located on the right hand side are infinitesimal.
Using the Banach--Steinhaus  Theorem lets make sure that 
$\left\{  \alpha^n ||(A+\alpha B)^{-1} B x_{n-1}||\right\}$  is infinitesimal.
Indeed, if  $B x_{n-1} \in \mathbb{R}(A),$
then there exist $x_n$ such as $A x_n = B x_{n-1}.$
 But in this case
 $\alpha^n (A+ \alpha B)^{-1} B x_{n-1} = \alpha^n (A+\alpha B)^{-1}(A+\alpha B -\alpha B)x_n = \alpha^n x_n -\alpha^{n+1} (A+\alpha B)^{1}  B x_n, $
where  $\alpha^{n+1} ||(A+\alpha B)^{1} B x_n || \leq \alpha^{n+1} c(\alpha) ||B x_n||,$
$\lim\limits_{\alpha \rightarrow 0} \alpha^{n+1} c(\alpha) = 0.$
Consequently, $\{ ||\alpha^n (A+\alpha B)^{1}  B x_{n-1} || \}$ is
infinitesimal, and the sequence of linear operators
$\{ \alpha^n (A+\alpha B)^{1}  B\}$
pointwise converges to zero operator on the linear manifold
$L = \{ x | B x \in {R}(A) \}$
and the sequence $\{|| \alpha^n (A+\alpha B)^{1} B ||\}$
 is bounded.
Since $I = \{ x \, |\, Bx \in \overline{{R}(A)} \}$
we complete the proof by the reference to the Banach--Steinhaus  Theorem.
\end{proof}
We apply   Th. 2 for construction of the  stable differentiation algorithm below.

\section{On differentiation regularization}

 Let $y: I \subset \mathbb{R} \rightarrow \mathbb{R}$   be continuous 
and differentiable function on the interval $I,$ and let
 its derivative  $y^{\prime}(t)$ be continuous  on the interval $(a,b) \subset I.$
Then $y(t) - y(+a) - y^{\prime}(+a)(t-a) = o(t-a)$ as $t \rightarrow +a.$
Let 
$\tilde{y}: \, [a,b] \rightarrow \mathbb{R}$ be a bounded function and values
  $c,\, d$ such as
$\sup\limits_{a<t<b} |\tilde{f}(t) - f(t)| = \mathcal{O}(\delta), $
where $f(t) = y(t) - y(+a) - y^{\prime}(+a)(t-a), $ $ \tilde{f} = \tilde{y}(t) - c - d(t-a).$
In applications the values
 $\tilde{y}(t_i)$  are usually know for $t_i = ih\in [a,b]$
such as $|y(t_i) - \tilde{y}(t_i)|=\mathcal{O}(\delta).$
 Our objective here is to find $\tilde{y}_i^{\prime}(t_i)$ with accuracy up to $\varepsilon$. 
Stable differentiation problem attracted many scientists including A.~N.~Tikhonov, V.~Ya.~Arsenin,  V.~V.~Vasin, V.~B.~Demidovich, T.~F.~Dolgopolova and V.~K.~Ivanov. Here readers may refer to textbook  \cite{tih1974}, p. 158 and \cite{hao2012} for
review of recent results .
Let us introduce the following equations:
 \begin{equation}
\int_a^t x(s)\, ds = f(t),
\label{eq31}
\end{equation}
$$ \int_a^t x_{\alpha}(s)\, ds + \alpha x_{\alpha}(t) = \tilde{f}(t), \,\,\,\,\,\,
 \int_a^t \tilde{x}_{\alpha}(s)\, ds + \alpha \tilde{x}_{\alpha}(t) = \tilde{f}(t).$$
Therefore here we have $A:=\int_a^t [\cdot]\, ds,$ 
${{R}(A)} = \left\{f(t)\in {\mathcal C}_{[a,b]}^{(1)}, \, f(+a)=0  \right\}=: \stackrel{\hspace{.9mm} \circ\,\,\,\,\,\,\,\,\,(1)\hspace{-6mm}}{\mathcal{C}}_{\hspace{-1mm}\,\,[a, b],},$\\
$\overline{{R}(A)} = \stackrel{\hspace{3.9mm} \circ\,\,\,\,\,\,\,\,\,}{\mathcal{C}}_{\hspace{-5mm}\,\,[a,\sc b],}$
 $B(\alpha):=\alpha I,$
$X=Y=\mathcal{C}_{[a,b]}.$  We construct the inverse operator  $(A+\alpha I)^{-1} \in \mathcal{L}
(\mathcal{C}_{[a,b]} \rightarrow \mathcal{C}_{[a,b]})$
explicitly as 
 $(A+\alpha I)^{-1} = \frac{1}{\alpha} - \frac{1}{\alpha^2} \int_a^t e^{ - \frac{t-s}{\alpha}}
[\cdot] \, ds. $
Since $f(t) \in {R}(A)$ then equation 
(\ref{eq31}) has unique solution
 $x^*(t) = A^{-1}f = y^{\prime}(t) - y^{\prime}(+a).$
It is to be noted that $$||(A+\alpha I)^{-1}||_{\mathcal{L}(\mathcal{C}_{[a,b]} \rightarrow 
\mathcal{C}_{[a,b]})}  \leq \frac{1}{\alpha} \left( 1+ \frac{1}{\alpha}\max\limits_{a\leq t \leq b}
\int_a^t e^{-\frac{t-s}{\alpha}} \, ds\right)
 = \frac{1}{\alpha}(2 - e^{-\frac{a-b}{\alpha}}) <\frac{2}{\alpha}.$$
Parameter $\alpha$ should be coordinated with
  $\delta,$ e.g. if  $\alpha =\sqrt{\delta}.$
Since $||B(\alpha)|| = \alpha,$ $c(\alpha)=\frac{2}{\alpha},$ $\mathcal{N}(A) = \{0\}$
then based on  Th. 2 continuous function $x^*(t) = y^{\prime}(t)-y^{\prime}(+a)$
is $B$-normal solution iff $x^*(t) \in \overline{{R}(A)}.$ 
In our case  ${R}(A) = \stackrel{\hspace{.9mm} \circ\,\,\,\,\,\,\,\,\,(1)\hspace{-6mm}}{\mathcal{C}}_{\hspace{-1mm}\,\,[a, b],}.$
Taking into account that linear functions space $\stackrel{\hspace{.9mm} \circ\,\,\,\,\,\,\,\,\,(1)\hspace{-6mm}}{\mathcal{C}}_{\hspace{-1mm}\,\,[a, b],}$ is dense  $L_1 = \{f(t) \in \mathcal{C}_{[a,b]}, \, f(+a)=0 \},$
$x^*(+a)=0,$ then  $x^*(t) \in \overline{{R}(A)}.$
 Therefore based on the Main Theorem and  Th. 2,
the  formula
\begin{equation}
\tilde{x}_{\alpha}(t) = \frac{\tilde{y}(t) - c - d(t-a)}{\alpha} 
- \frac{1}{\alpha^2} \int_a^t e^{-\frac{t-s}{\alpha}} \left (  \tilde{y}(s) - c - d(s-a) \right)\, ds 
\label{eq34}
\end{equation}
defines algorithm of  stable differentiation $\tilde{y}^{\prime}(t)$.
Or, more precisely $\forall \varepsilon >0$ $\exists \delta_0=\delta_0(\varepsilon)>0$
 such as, if
$\sup\limits_{a<t<b} |\tilde{f}(t) - f(t) | \leq \delta, \, \delta \leq \delta_0(\varepsilon)$
then $\max\limits_{a\leq t \leq b} \left | \tilde{x}_{\alpha}(t)- f^{\prime}(t) \right|\leq \varepsilon.$
If we select $\alpha = \sqrt{\delta},$ then 
$\lim\limits_{\delta \rightarrow 0} \max\limits_{a\leq t \leq b} |\tilde{x}_{\alpha}(t) - (y^{\prime}(t) - y^{\prime}(+a))| =0.$
Therefore, $\{ \tilde{x}_{\alpha}\}$ 
converges uniformly to
 $y^{\prime}(t) - y^{\prime}(+a)$ as $\delta \rightarrow 0.$
Based on  (\ref{eq34}) we constructed  regularized
differentiation algorithm, which is uniform w.r.t. $t \in [a,\, b]$.
Let us demonstrate its efficiency below.

\section{Numeric examples}


In this section we included two examples to demonstrate the efficiency of our approach.
We add noise to exact data as $\tilde{y}(t)=y(t)+\delta R(t)$ 
with the noise levels $\delta$=0.1, $\delta$=0.01 and $\delta$=0.001,
where $R(t)$ is a random function with zero mean value and standard deviation $\sigma =1$.
The number of grid points used is 512. Trapezoidal quadrature rule is used.

\begin{example} \label{example1}
For the first example we use the function
$
y(t)=\frac{1}{t^3+1}\sin\left(\frac{\pi t}{4}\right),\; t\in [0,\,3],
$
with its  derivative
$
{y}^{\prime}(t)=\frac{-3t^2}{{(t^3+1)}^{2}}\sin\left(\frac{\pi t}{4}\right)+\frac{\pi}{4(t^3+1)}\cos\left(\frac{\pi t}{4}\right).\,
$
Fig. 1
demonstrates exact and computed derivatives and the errors, noise level $\delta$=0.001.
The  maximum errors are given in Tab. 1.
\begin{figure}

\subfloat[]
{
    \includegraphics[scale=.62]{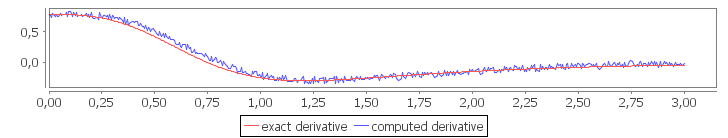}
    \label{fig:fig1}
}
\par
\subfloat[]
{
    \includegraphics[scale=.62]{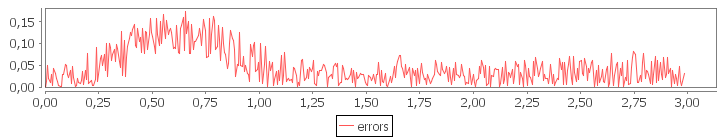}
    \label{fig:fig1}
}
\caption{Example 1. (a) the exact and the computed derivatives $\tilde{y}^{\prime}(t)$. (b) the errors.\\ The noise level $\delta$=0.001 is used  to generate $\tilde{y}(t).$}
\label{fig:foo1}
\end{figure}
%
%
\begin{figure}[htp!]
\begin{flushright}
\bf{Table 1}\vspace{-2mm}
\end{flushright}
\begin{center}\label{tab1}
\begin{tabular}{|c|c|c|c|}
\hline
\textbf{$\delta$} & 0.1 & 0.01 & 0.001\\ \hline
max error & 0{.}172977091 & 0{.}284849645 & 0{.}70584444\\ \hline
\end{tabular}
\end{center}
\end{figure}
\end{example}
\begin{example} \label{example2}
Here we used the exact function
$ y(t)=\cos(\frac{\pi t}{8})e^{-t^2},\; t\in [0,\,5], $
with its exact derivative
$ {y}^{\prime}(t)=-e^{-t^2}(\frac{\pi}{8}\sin(\frac{\pi t}{8})+2t \cos(\frac{\pi t}{8})).
$
The  maximum errors for this example are shown in Tab. 2.

\begin{figure}

\subfloat[]
{
    \includegraphics[scale=.62]{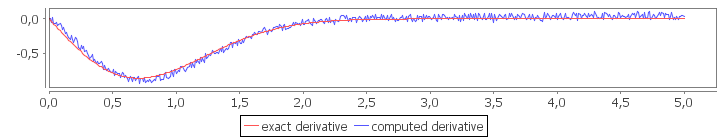}
    \label{fig:fig2}
}
\par
\subfloat[]
{
    \includegraphics[scale=.62]{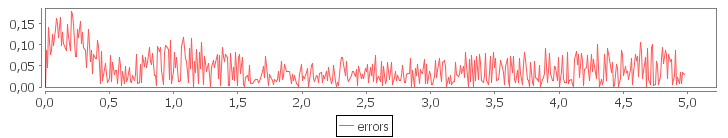}
    \label{fig:fig2}
}
\caption{ Example 2. (a) the exact and the computed derivatives $\tilde{y}^{\prime}(t).$ (b) the errors.\\ The noise level $\delta$=0.001 is used to generate $\tilde{y}(t).$}
\label{fig:foo2}
\end{figure}

%


\begin{figure}[htp!]
\begin{flushright}
\bf{Table 2}\vspace{-2mm}
\end{flushright}
\begin{center}
\label{tab2}
\begin{tabular}{|c|c|c|c|}
\hline
\textbf{$\delta$} & 0.1 & 0.01 & 0.001\\ \hline
max error & 0{.}169489112 & 0{.}277548421 & 0{.}637657088\\ \hline
\end{tabular}
\end{center}
\end{figure}

\end{example}

%



\newpage
\begin{thebibliography}{99}




\bibitem{hao2012}  H\`ao N. D., Chuonga L.H.,  Lesnic D.  Heuristic Regularization Methods for Numerical Differentiation. \textit{Computers and Mathematics with Applications},  2012,  vol. 63, p.~816--826.




\bibitem{iv1978} Ivanov V. K., Vasin V. V., Tanana V. P. \textit{The Theory of Linear Ill-posed Problems and their Applications} (in Russian). Moscow, Nauka, 1978.


\bibitem{lat1970} Lattes  R., Lions J. L. \textit{The Method of Quasi-Reversibility.
Applications to Partial Differential Equations,} American Elsevier Publishing Company. 1969.


\bibitem{lav1962} Lavrentiev  M. M. \textit{Some Improperly Posed Problems
in Mathematical Physics.} Springer, Berlin, 1967.


\bibitem{leon13} Leontiev R.Yu. \textit{Nonlinear Equations in Banach Spaces with Vector Parameter
in Singular Case.} (in Russian). Irkuts State University Publ., 2013, 101 p..


\bibitem{log1976} Loginov B. V., Sidorov N.A.  Calculation of Eigenvalues and Eigenvectors
of Bounded Operators by the False-Perturbation Method. 
\textit{Mathematical notes of the Academy of Sciences of the USSR}, 1976, vol. 19, issue~1, p.~62--64. 


\bibitem{mar1973}  Marchuk G. I.  Perturbation Theory and the Statement of Inverse Problems. \textit{Lecture Notes in Computer Science.} Vol. 4: 5th Conf. on Optimization 
Tech. Springer, 1973,  p. 159--166. 

\bibitem{mas1968} Maslov V.P.
The Existence of a Solution of an Ill-Posed Problem is 
Equivalent to the Convergence of a Regularization Process. (in Russian)
\textit{Uspekhi Mat. Nauk}, 1968, vol. 23, no 3(141), p. 183--184. 






\bibitem{sid2014} 
 Sidorov D. Integral Dynamical 
{Integral Dynamical Models:
Singularities, Signals and Control}. Ed. by L. O. Chua. Singapore, London: World Scientific Publ.,
2014. Vol. 87 of {\it World Scientific Series on Nonlinear Science, Series A}, 243~p.


\bibitem{sid1980} Sidorov N. A., Trenogin V.A.  Linear Equations Regularization using
the Perturbation Theory. \textit{ Diff. Eqs.}, 1980, vol. 16, no.~11, p.~2038--2049. 

\bibitem{sid1978}  Sidorov N. A. 
Calculation of Eigenvalues and Eigenvectors
of Linear Operators by the Theory of Perturbations. (in Russian)
\textit{Differential Equations}, 1978, vol. 14, no~8, p.~1522--1525. 




\bibitem{sid1982} Sidorov N. A. \textit{General Issues of Regularization
in the Problems of the Theory of Branching.}  Irkutsk State University Publ., Irkutsk, 1982, 312 p.




\bibitem{sid1995} Sidorov N.A.  Explicit and Implicit Parametrisation
of the Construction of Branching Solutions by Iterative Methods. \textit{Sbornik: Mathematics.}
1995, vol. 186, no. 2, pp. 297--310.

\bibitem{dreg2012} Sidorov N.A., Leont'ev R.Yu., Dreglya A.I.
On Small Solutions of Nonlinear Equations with Vector Parameter in Sectoral Neighborhood.
\textit{Mathematical Notes,} feb. 2012, vol. 91, no. 1-2, p. 90-104.

\bibitem{sid2010a}  {Sidorov N. A., Sidorov D. N.}
Solving the Hammerstein Integral Equation in Irregular Case by Successive 
Approximations. \textit{Siberian Mathematical Journal}, March 2010, vol. 51, no. 2, p. 325--329.

\bibitem{sid2010}  {Sidorov N. A., Sidorov D. N., Krasnik A.V.}
On Solution of the Volterra Operator-Integral Equations in Irregular Case using Successive
Approximations. \textit{Differential Equations.} 2010, vol. 46, no. 6, p.874--882.



\bibitem{sid2002}  Sidorov N., Loginov B., Sinitsyn A., Falaleev M. \textit{Lyapunov-Schmidt Methods in Nonlinear 
Analysis and Applications.} Dortrecht, Kluwer Academic Publ., 2002, 548 p.


\bibitem{sid1976}  Sidorov N. A., Trenogin V. A.  A Certain Approach 
the Problem of Regularization of the Basis of the Perturbation of Linear 
Operators.  \textit{Mathematical notes of the Academy of Sciences of the USSR}, 1976, vol. 20, no.~5, p.~976--979.

\bibitem{bak1968} 
Sizikov V. S.  {Further Development of the New Version of a Posteriori Choosing Regularization Parameter in Ill-Posed Problems}, \textit{Intl. J. of Artificial Intelligence,} 2015, vol. 13, no. 1, p. 184--199.

\bibitem{stech1967} Stechkin S.B. The Best Approximation
of Linear Operators. \textit{Mat. Notes.} 1967, vol. 1, no 2, p. 137--148. 


\bibitem{tih1974} Tikhonov A. N., Arsenin  V. Ya. \textit{Solutions of Ill-Posed Problems.} Wiley. New York, 1977.

\bibitem{trenogin} Trenogin V. A. Functional Analysis Nauka. Moscow, 1980, 496~p.

\bibitem{tren1977}  Trenogin V. A., Sidorov D.N. Regularization of Computation of
Branching Solution of Nonlinear Equations. \textit{Lecture Notes in Mathematics.}
 1977,  vol. 594, p.~491--506. 




\bibitem{yag2014} Yagola  A. G.  \textit{Inverse Problems and Methods 
of Their Solution. Applications to Geophysics.} (in Russian) Binom Publ. Ser. Mathematical Modelling,
2014, 216 p.









\end{thebibliography}
\end{document}